\documentclass[a4paper,12pt,reqno]{amsart}
\usepackage{amsmath,amsfonts,amssymb,amsthm,enumerate,multicol}
\usepackage{tikz,graphicx}
\usepackage{stackengine}
\usepackage{subfigure,hyperref}
\usepackage{caption,enumitem,wasysym}
\usepackage{amsmath,amssymb}
\usepackage{cleveref}
\usepackage{mathrsfs}
\usepackage{physics}
\usepackage{graphics,graphicx}
\usepackage{stackengine}
\usepackage{bigints}
\usepackage{enumitem}
\usepackage{geometry}
\usepackage{epstopdf}
\usepackage{epsfig}
\usepackage{caption}
\usepackage{mathtools}
\usepackage{ragged2e}
\usepackage[super]{nth}
\usepackage{float,wrapfig}
\usepackage[labelsep=space]{caption}
\usepackage[font=footnotesize]{caption}
\usepackage{xcolor}
\newtheorem{lemma}{Lemma}[section]
\newtheorem{theorem}{Theorem}[section]
\newtheorem{proposition}{Proposition}[section]
\newtheorem{definition}{Definition}[section]

\newtheorem{corollary}{Corollary}[section]
\newtheorem{remark}{Remark}[section]

\numberwithin{equation}{section}

\title{Complete Monotonicity of the function involving derivatives of Barnes G-function}
\author{Deepshikha Mishra$^\dagger$}
\address{$^\dagger$Department of Mathematics\\ Indian Institute of Technology, Roorkee-247667, Uttarakhand, India}
\email{deepshikha\_m@ma.iitr.ac.in, deepshikhamishraalld@gmail.com}
\author[A. Swaminathan]{A. Swaminathan\, $^{\#\ddagger}$}{\thanks{$^{\#}$Corresponding author}}
\address{$^\ddagger$Department of Mathematics\\ Indian Institute of Technology, Roorkee-247667, Uttarakhand, India}
\email{a.swaminathan@ma.iitr.ac.in, mathswami@gmail.com}
\allowdisplaybreaks
\bigskip

\begin{document}
	
	
\begin{abstract}
In this manuscript, we present the complete monotonicity of functions defined in terms of the poly-double gamma function
\begin{align*}
	\psi_2^{(n)}(x) = (-1)^{n+1} n! \sum_{k=0}^{\infty} \dfrac{(1+k)}{(x+k)^{n+1}}, \quad x > 0, \ n\geq 2.
\end{align*}
Consequently, we derive bounds for the ratio involving $\psi_2^{(n)}(x)$ and apply these bounds to establish the convexity, subadditivity and superadditivity of $\psi_2^{(n)}(x)$. In the process, various fundamental properties of $\psi_2^{(n)}(x)$ are established, including recurrence relations, integral representations, asymptotic expansions, complete monotonicity, and related inequalities. Graphical illustrations are provided to support the theoretical results.
\end{abstract}

\maketitle
\markboth{Deepshikha Mishra and A. Swaminathan}{Complete Monotonicity of the function involving derivatives of Barnes G-function}


\section{Introduction and preliminaries}\label{Introduction and preliminaries}
The Barnes G-function, denoted as $G(z)$, was first introduced by Barnes in \cite{Barnes_1899_genesis of double gamma}. It satisfies the recurrence relation \cite{Barnes_1899_genesis of double gamma}
\begin{align*}
G(z+1) &= \Gamma(z) G(z), \quad z \in \mathbb{C},
\end{align*}
with
\begin{align*}
G(1) = 1,
\end{align*}
where $ \Gamma(z) $ denotes the classical Euler gamma function. The Barnes G-function, commonly referred to as the double gamma function, corresponds to the product of Euler gamma functions given by
\begin{align*}
G(n+1) = \prod_{k=1}^{n-1} \Gamma(k+1), \quad n \in \mathbb{N}.
\end{align*}
The Barnes G-function, $G(z)$, also possesses a Weierstrass representation \cite{Barnes_1899_genesis of double gamma, Ferreira_2001_asym of double gamma}, given by
\begin{align}\label{barnes g function definition}
G(z+1) &= (2 \pi)^{z/2} e^{-\left((1+\gamma)z^2+z\right)/2} \prod_{k=1}^{\infty} \left[ \left(1 + \dfrac{z}{k} \right)^k e^{-z + z^2/2k} \right], \quad \forall \ z\in \mathbb{C},
\end{align}
where $ \gamma \approx 0.5772 $ represents the Euler constant. Later, $\Gamma_2(z)=\dfrac{1}{G(z)}$ was generalized to the multiple gamma function $\Gamma_p(z)$, which is defined by the following recurrence and functional equations \cite{Choi_multiple gamma_amc_2003}
\begin{align*}
&\Gamma_{p+1}(z+1) = \dfrac{\Gamma_{p+1}(z)}{\Gamma_p(z)}, \quad z \in \mathbb{C}, \quad p \in \mathbb{N},\\
&\Gamma_1(z) = \Gamma(z), \quad \Gamma_p(1) = 1,
\end{align*}
along with the convexity condition
\begin{align*}
	&(-1)^{p+1} \dfrac{d^{p+1}}{dx^{p+1}} \log \Gamma_p(x) \geq 0, \quad x > 0.
\end{align*}
The multiple gamma function has applications in various fields, including mathematical physics \cite{Quine_1996_gamma2 on mathematical phy}, analytic number theory \cite{Adamchik_1998_zeta_analysis}, and probability theory \cite{Nikeghbali_2009_barnes g in probability}, among others. In this manuscript, we focus on the double gamma function and its logarithmic derivatives.

Let $\psi^{(n)}_p(x) $ denotes the $(n+1)$th logarithmic derivative of the multiple gamma function $ \Gamma_p(x) $, we call it as the \textit{poly-multiple gamma function}. For any $ n, p \in \mathbb{N}$ with $ n \geq p $, $ \psi^{(n)}_p(x) $ is defined as \cite{Das_2018_cr acad}
\begin{align*}
\psi_p^{(n)}(x)= (-1)^{n+1} \sum_{k=0}^{\infty} \binom{p+k-1}{p-1} \dfrac{n!}{(x+k)^{n+1}}, \quad x > 0.
\end{align*}
For $ p=2 $, we obtain  $\psi_2^{(n)}(x)$ such that
\begin{align}\label{poly-double gamma function series form}
\psi_2^{(n)}(x)=  (-1)^{(n+1)} n! \sum_{k=0}^{\infty} \dfrac{(1+k)}{(x+k)^{n+1}}, \quad x > 0, \ n\geq 2.
\end{align}
We refer to $\psi_2^{(n)}(x)$ as the \textit{poly-double gamma function}. This function is also related to the polygamma function $\psi^{(n)}(x)$ defined as
\begin{align*}
\psi^{(n)}(x) = (-1)^{n+1} n! \sum_{k=0}^{\infty} \frac{1}{(k+x)^{n+1}}, \quad n\geq 1, \ \ x>0
\end{align*}
by the following relation
\begin{align*}
\psi_2^{(n)}(x) = -n \psi^{(n-1)}(x) + (1 - x) \psi^{(n)}(x), \quad n\geq 2, \ x>0.
\end{align*}
Since $\psi_2^{(n)}(x)$ is defined for $ n \geq 2 $, we can not obtain $\psi_2(x)$ directly from \eqref{poly-double gamma function series form}. To define $\psi_2(x)$, we take the logarithmic derivative of \eqref{barnes g function definition} which leads to
\begin{align*}
\psi_2(x) &= -\dfrac{1}{2} \log(2 \pi) + (1+\gamma)x +\dfrac{1}{2}- \sum_{k=0}^{\infty}  \dfrac{(x-1)^2}{(k+1)(x+k)}, \quad x > 0.
\end{align*}
We refer to $\psi_2(x)$ as the \textit{di-double gamma function}. Recently, various properties of the poly-multiple gamma function, including monotonicity, convexity, and subadditivity, are established, along with some inequalities and zeros-related results \cite{Das_2018_cr acad, Mezo_2017_Zeros of gamma2 itsf}.


Now we present some definitions and theorems that will be used later. We begin with Lagrange's identity, which establishes a connection between the product of two sums of squares and the sum of pairwise products. The Lagrange’s inequality \cite{Mitrinovic_1970_Analytic ineq} for any two real sequences $\{a_n\}$ and $\{b_n\}$ is given as
\begin{align}\label{lagrange identity}
\left( \sum_{i=1}^{n} a_i^2 \right) \left( \sum_{i=1}^{n} b_i^2 \right) - \left( \sum_{i=1}^{n} a_i b_i \right)^2 = \sum_{1 \leq i < j \leq n} (a_i b_j - a_j b_i)^2.
\end{align}
This identity also provides a refinement of the Cauchy–Schwarz inequality.
\begin{definition}\cite{Donoghue_monotone matrix func}\label{cmf definition}
Let $g(x)$ be a real-valued and $C^{\infty}$ function on the half-axis $x>0$. Then we say that $g(x)$ is completely monotone if for all $k \geq 0$
\begin{align*}
(-1)^k g^{(k)}(x) \geq 0
\end{align*}
on the half-axis.
\end{definition}
Note that, the completely monotone functions are also characterized by the Bernstein theorem, as stated below.
\begin{theorem}\cite{Schilling_2012_bernstein func}\label{Bernstein theorem}
A real-valued function $g(x)$ is completely monotonic in half-axis $x>0$ if, and only if, it can be represented as
\begin{align}\label{bernstein eq}
g(x) = \int_{[0, \infty)} e^{-x t}  d\mu(t),
\end{align}
on half axis $x>0$.
\end{theorem}
\begin{remark}\label{scm remark}
If the measure $\mu$ defined in the integral \eqref{bernstein eq} is absolutely continuous, i.e., if
\begin{align*}
d\mu(t) = m(t) dt,
\end{align*}
where $m(t)$ is a non-negative, non-decreasing function, then the function $g(x)$ exhibits strong complete monotonicity rather than complete monotonicity.
\end{remark}

Complete monotone functions are closely connected to Pick functions, Bernstein functions, and Stieltjes functions \cite{Schilling_2012_bernstein func}, and play a key role in classifying various classes of special functions \cite{Berg_2024_cbf related to gamma, Berg_2001_cm related to gamma, Das_swami_2017_Pick function, Pedresan_2009_cm of gamma2 jmaa, Pedersen_2003_Pickfunc gamma2}. One of the key results on complete monotonicity involving the polygamma function was established in \cite{Alzer_1998_Inequality polygamma_siam}, while similar kinds of generalizations have been studied recently in \cite{Liang_cm of polygamma_JIA, Zhang_2020_cm of k-polygamma_JIA}. This naturally raises the question of whether such monotonicity properties can be extended to functions involving higher-order gamma functions, particularly the poly-double gamma function. The primary aim of this manuscript is to explore this aspect of $\psi_2^{(n)}(x)$. To address this, we first discuss some fundamental results related to the polygamma function.


The polygamma function $\psi^{(n)}(x)$ satisfies the following recurrence relation \cite{Abramowitz_handbook special func}
\begin{align} \label{recurrence relation for polygamma}
\psi^{(n)}(x+1) = \psi^{(n)}(x) + (-1)^n \dfrac{n!}{x^{n+1}}, \quad n \in \mathbb{N}\cup \{0\}, \quad x > 0.
\end{align}
In addition to the recurrence relation \eqref{recurrence relation for polygamma}, the asymptotic nature of the polygamma function as $x \to \infty$ is given by \cite{Abramowitz_handbook special func}
\begin{align}\label{asymptotic expansion for polygamma}
\psi^{(n)}(x) \sim (-1)^{n-1} \left[\dfrac{(n-1)!}{x^n}+ \dfrac {n!}{2x^{n+1}}+\sum_{k=1}^{\infty} B_{2k} \dfrac{(2k+n-1)!}{(2k)! x^{2k+n}} \right],
\end{align}
where $B_{2k}$ are Bernoulli numbers and $n\in \mathbb{N}$. As a result of \eqref{asymptotic expansion for polygamma}, we obtain
\begin{align}\label{asymptotic expansion for polygamma limit form}
\lim_{x \to \infty} x^n \psi^{(n)}(x) = (-1)^{n-1}(n-1)!.
\end{align}
The manuscript is organized as follows: we begin with Section \ref{some results} which examines fundamental properties of the poly-double gamma function, including recurrence relation, integral representation, asymptotic expansion, complete monotonicity and related inequalities. In Section \ref{complete monotonicity}, we state and prove the main theorem regarding the complete monotonicity of functions written in terms of the poly-double gamma function along with its consequences.

\section{Inequalities involving Poly-double gamma function}\label{some results}
In this section, we obtain an integral representation of $\psi_2^{(n)}(x)$ and, based on this representation, we examine its complete monotonicity along with various related inequalities, including Turan-type inequality. We also obtain the asymptotic expansion of $\psi_2^{(n)}(x)$ and determine the limiting value of $x^{n-1} \psi_2^{(n)}(x)$ by incorporating the recurrence relation for $\psi_2^{(n)}(x)$.

\begin{theorem}\label{poly double cm thm}
The poly-double gamma function $\psi_2^{(n)}(x)$ admits the following integral representation
\begin{align}\label{integral form of poly-double gamma function}
\psi_2^{(n)}(x) = (-1)^{n+1} \int_0^\infty e^{-xt} \dfrac{t^n}{(1 - e^{-t})^2} dt, \quad n \geq 2, \quad x > 0.
\end{align}
Moreover, for every integer $n \geq 2$, the function $(-1)^{n+1} \psi_2^{(n)}(x)$ is completely monotone on $(0, \infty)$ and, therefore, it is also convex.
\end{theorem}
\begin{proof}
In \cite{Vigneras_1979_gamma2 integral}, the following integral form is given by
\newline
$
\displaystyle
\log \Gamma_2(x+1)
$
\begin{align}\label{integral form of log double gamma function}
= - \int_0^\infty \! \dfrac{e^{-t}}{t(1 - e^{-t})^2} \left( 1 - xt - \dfrac{x^2 t^2}{2} - e^{-xt} \right) dt + (1 + \gamma) \dfrac{x^2}{2} - \dfrac{3}{2} \log \pi, \quad x > -1.
\end{align}
Differentiating \eqref{integral form of log double gamma function} with respect to $x$ by applying Leibniz's rule for differentiation under the integral sign, and substituting $\dfrac{\Gamma_2^{'}(x+1)}{\Gamma_2(x+1)} = \psi_2(x+1)$, we obtain
\begin{align}\label{di double gamma function}
\psi_2(x+1) = \int_0^\infty \dfrac{ e^{-t}}{t(1 - e^{-t})^2} (t + xt^2 - te^{-xt}) dt + (1 + \gamma)x, \quad x > -1.
\end{align}
Differentiating \eqref{di double gamma function} with respect to $x$ upto $n$ times, and then replacing $x \to x-1$, gives \eqref{integral form of poly-double gamma function}.

Using \eqref{integral form of poly-double gamma function}, we can write
\begin{align}\label{poly double gamma in terms of lapl transform eq}
(-1)^{n+1} \psi_2^{(n)}(x) = \int_0^\infty e^{-xt} \dfrac{t^n}{(1 - e^{-t})^2} dt, \quad n \geq 2, \quad x > 0.
\end{align}
Since $\dfrac{t^n}{(1 - e^{-t})^2}$ is positive for all $t \in (0,\infty)$ and $n \geq 2$, Theorem \ref{Bernstein theorem} asserts the complete monotonicity of $(-1)^{n+1} \psi_2^{(n)}(x)$. Also, it is well known that complete monotonicity of a function implies its log-convexity \cite{Wimp_1981_cm in numerical}. Therefore, $(-1)^{n+1} \psi_2^{(n)}(x)$ is log-convex on $(0, \infty)$ for every $n \in \mathbb{N}$ with $n \geq 2$ and hence it is convex.
\end{proof}
The following graphical representation demonstrates the complete monotonic nature of $(-1)^{n+1} \psi_2^{(n)}(x)$ for $n \geq 2$ on $(0,\infty)$.
\begin{figure}[H]
	\footnotesize
	\stackunder[5pt]{\includegraphics[scale=0.9]{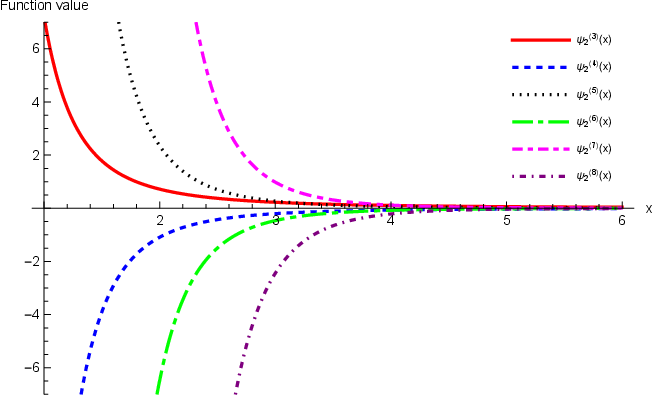}}{Graph of $\psi_2^{(3)}(x)$ up to its $5$th derivative}
	\caption{Complete monotonicity of $(-1)^{n+1} \psi_2^{(n)}(x)$ for $n=3$}
	\label{fig:Complete monotonicity of psi function}
\end{figure}
For $n=3$, the function $(-1)^{n+1} \psi_2^{(n)}(x)$ simplifies to $\psi_2^{(3)}(x)$. As shown in Figure \ref{fig:Complete monotonicity of psi function}, $\psi_2^{(3)}(x)$ remains positive, and its derivatives alternate in sign, beginning with a negative value. This behaviour confirms the complete monotonicity of $\psi_2^{(3)}(x)$. We have illustrated this up to the $5$th derivative. Higher-order derivatives can also be shown to follow this pattern.
\par
As we have shown in Theorem \ref{poly double cm thm} that $\psi_2^{(n)}(x)$ is completely monotone, now we analyze its strong complete monotonic behaviour, and for that, let us assume
\begin{align*}
	\tilde{m}(t) = \frac{t^n}{(1 - e^{-t})^2}, \quad t > 0, \quad n\geq 2.
\end{align*}
It is clear that $\tilde{m}(t)$ is non-negative for all $t > 0$. Also, differentiating $\tilde{m}(t)$ with respect to $t$, we obtain
\begin{align*}
	\tilde{m}'(t) = \frac{n e^{-t} t^{n-1}}{(1 - e^{-t})^3} \left(e^t - \left(1 + \frac{2t}{n}\right)\right),\quad n\geq 2.
\end{align*}
Since
\begin{align*}
	e^t > 1 + t > 1 + \frac{2t}{n}, \qquad \forall \quad n\geq 2, \ \ t>0,
\end{align*}
we have, $\tilde{m}'(t) > 0$, which implies, $\tilde{m}(t)$ is a non-negative and non-decreasing function.

Note that $\tilde{m}(t)$ represents the density of the measure given in integral \eqref{poly double gamma in terms of lapl transform eq}. Hence, by using Remark \ref{scm remark}, the function $(-1)^{n+1} \psi_2^{(n)}(x)$ is not only completely monotone but also strongly completely monotone.

\begin{corollary}
Let $x > 0$ be a real number and $n \geq 2$ an integer. Then $\psi_2^{(n)}(x)$ satisfies the Turan-type inequality
\begin{align}\label{turan inequality of poly-double gamma function}
\left(\psi_2^{(n)}(x+1)\right)^2 &\leq \psi_2^{(n)}(x) \psi_2^{(n)}(x+2).
\end{align}
\end{corollary}
\begin{proof}
Since completely monotonic function is also logarithmically convex \cite{Widder_1941_The Laplace Transform}, we have
\begin{align*}
(-1)^{n+1} \psi_2^{(n)}(a x+ (1-a)y) &\leq \left((-1)^{n+1} \psi_2^{(n)}(x)\right)^a \left((-1)^{n+1} \psi_2^{(n)}(y)\right)^{1-a},
\end{align*}
for any $a \in [0,1]$ and $x, y > 0$. In particular, setting $a = \dfrac{1}{2}$ and $y = x + 2$ yields the Turan-type inequality \eqref{turan inequality of poly-double gamma function}.
\end{proof}
The graphical representation of \eqref{turan inequality of poly-double gamma function} is given below.
\begin{figure}[H]
	\footnotesize
	\stackunder[5pt]{\includegraphics[scale=0.8]{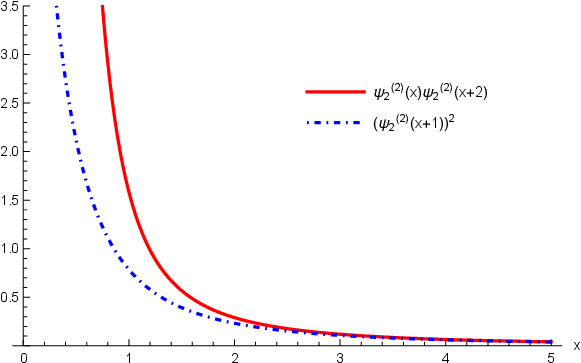}}{Graph of $\psi_2^{(n)}(x+1)$ and $\psi_2^{(n)}(x) \psi_2^{(n)}(x+2)$ for $n=2$}
	\caption{Turan-type inequality for $(-1)^{n+1} \psi_2^{(n)}(x)$ for $n=2$}
	\label{fig:Turan inequality}
\end{figure}
As shown in Figure \ref{fig:Turan inequality}, the red (solid) curve corresponding to $\psi_2^{(n)}(x) \psi_2^{(n)}(x+2)$ lies entirely above the blue (dot-dashed) curve representing $\left(\psi_2^{(n)}(x+1)\right)^2$, illustrating the Turan-type inequality \eqref{turan inequality of poly-double gamma function}.


In addition to the Turan-type inequality, which follows from the completely monotonic nature of $(-1)^{n+1} \psi_2^{(n)}(x)$, the sub-additive (super-additive) properties of $\psi_2^{(n)}(x)$ are also consequences of its completely monotonic behaviour. In \cite{Alzer_analysis mathematica_2025}, these properties are examined for trigonometric functions using their complete monotonicity. The next result establishes the subadditive (superadditive) properties of $\psi_2^{(n)}(x)$.
\par

\begin{theorem}\label{sub and super additivity thm}
Let $n$ and $r$ be non-negative integers with $n \geq 2$ and $r \geq 0$. Assume that $x_1, x_2 \in (0, m)$ such that  $x_1+x_2 \leq m$. Then the following inequalities hold.
\begin{enumerate}
\rm \item If $n$ and $r$ differ in parity, that is, one is even and the other is odd, then
\begin{align}\label{subadditivity of polydouble gamma}
\psi_2^{(n+r)}(x_1+x_2) < \psi_2^{(n+r)}(x_1) + \psi_2^{(n+r)}(x_2) .
\end{align}
\rm \item If $n$ and $r$ are of the same parity, that is, both even or both odd then
\begin{align}\label{superadditivity of polydouble gamma}
\psi_2^{(n+r)}(x_1+x_2)> \psi_2^{(n+r)}(x_1) + \psi_2^{(n+r)}(x_2)  .
\end{align}
\end{enumerate}
\end{theorem}
Note that we consider $r=0$ as an even number.
\begin{proof}
Consider the case that $r$ is an odd number. Define the function $P_n(x_1, x_2)$ by
\begin{align}\label{value of P_n(x,y)}
P_n(x_1, x_2) = \left(g_n(x_1)\right)^{(r)} + \left(g_n(x_2)\right)^{(r)} - \left(g_n(x_1 + x_2)\right)^{(r)},
\end{align}
where
\begin{align}\label{value of g_n}
g_n(x_1) = (-1)^{n+1} \psi_2^{(n)}(x_1).
\end{align}
Differentiating $P_n(x_1, x_2)$ with respect to $x_1$ gives
\begin{align*}
\dfrac{\partial}{\partial x_1} P_n(x_1, x_2) = \left(g_n(x_1)\right)^{(r+1)} - \left(g_n(x_1 + x_2)\right)^{(r+1)}.
\end{align*}
Since $g_n(x_1)$ is completely monotonic, $\left(g_n(x_1)\right)^{(r+1)}$ is decreasing when $r$ is odd. It follows that $P_n(x_1,x_2)$ is an increasing function of $x_1$. Thus, from \eqref{value of P_n(x,y)}, we can write
\begin{align*}
P_n(x_1, x_2) \leq \left(g_n(m - x_2)\right)^{(r)} + \left(g_n(x_2)\right)^{(r)} - \left(g_n(m)\right)^{(r)}.
\end{align*}
Moreover, $\left(g_n(x_2)\right)^{(r)}$ is increasing for odd $r$ due to complete monotonicity of $\left(g_n(x_2)\right)$, and $\left(g_n(m - x_2)\right)^{(r)}$ is negative for odd $r$ due to the complete monotonicity of $g_n(x_1)$, we conclude that
\begin{align*}
P_n(x_1, x_2) \leq 0.
\end{align*}
Using \eqref{value of P_n(x,y)} and \eqref{value of g_n}, we obtain
\begin{align*}
\left((-1)^{n+1} \psi_2^{(n)}(x_1)\right)^{(r)} + \left((-1)^{n+1} \psi_2^{(n)}(x_2)\right)^{(r)} - \left((-1)^{n+1} \psi_2^{(n)}(x_1 + x_2)\right)^{(r)} \leq 0.
\end{align*}
Hence, when $n$ is even, the above inequality gives \eqref{subadditivity of polydouble gamma}, and when $n$ is odd, it gives \eqref{superadditivity of polydouble gamma}. The case for even $r$ follows in a similar way.
\end{proof}
It is important to note that we can sharpen inequalities \eqref{subadditivity of polydouble gamma} and \eqref{superadditivity of polydouble gamma} in the following manner.
\begin{theorem}
Let $n, r$ be non-negative integers with $n \geq 2$ and $r \geq 0$. Then, for $x_1, x_2 \in (0, m)$ such that $x_1+x_2 \leq m$, the following inequalities hold.
\begin{enumerate}
\rm \item If $n$ and $r$ are differ in parity that is, one is even and the other is odd then
\begin{align}\label{subadditivity sharp}
\psi_2^{(n+r)}(x_1+x_2) - \psi_2^{(n+r)}(x_1) -\psi_2^{(n+r)}(x_2) \leq  2 \psi_2^{(n+r)}\left(\dfrac{m}{2}\right)-\psi_2^{(n+r)}(m).
\end{align}
\rm \item If $n$ and $r$ are of the same parity that is, both even or both odd then
\begin{align}\label{superadditivity sharp}
\psi_2^{(n+r)}(x_1+x_2) - \psi_2^{(n+r)}(x_1) -\psi_2^{(n+r)}(x_2) \geq  2 \psi_2^{(n+r)}\left(\dfrac{m}{2}\right)-\psi_2^{(n+r)}(m).
\end{align}
\end{enumerate}
\end{theorem}

Note that we consider $r=0$ as an even number.

\begin{proof}
In the proof of Theorem \ref{sub and super additivity thm}, we established that for odd values of $r$, the function $P_n(x_1, x_2)$ is increasing in $x_1$, which gives
\begin{align}\label{P_n lessthan P_n tilde}
P_n(x_1, x_2) \leq \tilde{P}_n(x_2),
\end{align}
where
\begin{align*}
\tilde{P}_n(x_2) = \left(g_n(m - x_2)\right)^{(r)} + \left(g_n(x_2)\right)^{(r)} - \left(g_n(m)\right)^{(r)}.
\end{align*}
Differentiating $\tilde{P}_n(x_2)$ with respect to $x_2$ yields
\begin{align*}
\tilde{P}_n'(x_2) = -\left(g_n(m - x_2)\right)^{(r+1)} + \left(g_n(x_2)\right)^{(r+1)}.
\end{align*}
Since $g_n$ is completely monotonic, it follows that $\left(g_n(x_2)\right)^{(r+1)}$ is decreasing on the interval $(0, m)$. Consequently, we obtain
\begin{align*}
\left(g_n(x_2)\right)^{(r+1)} & \geq \left(g_n(m - x_2)\right)^{(r+1)}\quad  \forall \ x_2 \in\left(0,\dfrac{m}{2}\right]\\
\left(g_n(x_2)\right)^{(r+1)} & \leq \left(g_n(m -x_2)\right)^{(r+1)} \quad  \forall \ x_2  \in\left[\dfrac{m}{2},m\right).
\end{align*}
This implies that $\tilde{P}_n(x_2)$ is increasing on $\left(0, \dfrac{m}{2}\right]$ and decreasing on $\left[\dfrac{m}{2}, m\right)$, and thus attains its maximum at $x_2 = \dfrac{m}{2}$. Therefore, it follows from inequality \eqref{P_n lessthan P_n tilde} that
\begin{align*}
P_n(x_1, x_2) \leq \tilde{P}_n(x_2) \leq \tilde{P}_n\left(\dfrac{m}{2}\right)
\end{align*}
which yields the result stated in \eqref{subadditivity sharp}. The proof of \eqref{superadditivity sharp} follows in a similar manner.
\end{proof}


\begin{remark}\label{polydouble zeta relation}
The bounds on the right-hand sides of inequalities \eqref{subadditivity sharp} and \eqref{superadditivity sharp} can be further rewritten in terms of $\zeta(s, a)$ and $\psi^{(n)}(x)$ by applying the following identities
\begin{align*}
\psi_2^{(n+r)}\left(\dfrac{m}{2}\right) - \psi_2^{(n+r)}(m) &= (-1)^{n+1} (n+r)! \left(2\zeta\left(n+r, \dfrac{m}{2}\right) + (2 - m) \zeta\left(n+r+1, \dfrac{m}{2}\right)\right) \\
&\quad + (-1)^{n-1} (n+r)! \left(\zeta(n+r, m) + (1 - m) \zeta(n+r+1, m)\right),
\end{align*}
and
\begin{align*}
\psi_2^{(n+r)}\left(\dfrac{m}{2}\right) - \psi_2^{(n+r)}(m) &= -2(n+r) \psi^{(n+r-1)}\left(\dfrac{m}{2}\right) + (2 - m) \psi^{(n+r)}\left(\dfrac{m}{2}\right) \\
&\quad + (n+r) \psi^{(n+r-1)}(m) - (1 - m) \psi^{(n+r)}(m),
\end{align*}
where $\zeta(s, a)$ and $\psi^{(n)}(x)$ denote the Hurwitz zeta function and the polygamma function, respectively, defined as
\begin{align*}
\zeta(s, a) = \sum_{k=0}^{\infty} \dfrac{1}{(k + a)^s}, \quad  a > 0, \  s > 1,
\end{align*}
and
\begin{align*}
\psi^{(n)}(x) = (-1)^{n+1} n! \sum_{k=0}^{\infty} \dfrac{1}{(x + k)^{n+1}}, \quad x > 0, \   n \geq 1.
\end{align*}
\end{remark}
Remark \ref{polydouble zeta relation} suggests that similar identities involving other special functions can be obtained. Such development may lead to further analysis, providing scope for future research in this direction. Next, we write the recurrence relation for $\psi_2^{(n)}(x)$, which will help us to prove Theorem \ref{limit formula thm}.
\begin{lemma}
Let $n$ be a non-negative integer. The recurrence relation for the poly-double gamma function $\psi_2^{(n)}(x)$ is given by
\begin{align}\label{recurrence relation of poly-double gamma function}
\psi_2^{(n)}(x+1) + \psi^{(n)}(x) = \psi_2^{(n)}(x), \quad  x > 0.
\end{align}
\end{lemma}
\begin{proof}
The recurrence relation for the double gamma function is \cite{Barnes_1899_genesis of double gamma}
\begin{align*}
\Gamma_2(x+1) \Gamma(x) = \Gamma_2(x), \quad x>0.
\end{align*}
Taking the logarithmic derivative on both sides with respect to $x$ yields
\begin{align}\label{digamma and di double gamma function}
\psi_2(x+1) + \psi(x) = \psi_2(x),
\end{align} 	
where $\psi(x)$ is the digamma function, and $\psi_2(x)$ denotes the logarithmic derivative of $ \Gamma_2(x)$, referred to as the \textit{di-double gamma function}.\\
Further differentiation of \eqref{digamma and di double gamma function} up to $n+1$ times leads to \eqref{recurrence relation of poly-double gamma function}.
\end{proof}

%


Based on the asymptotic expansion of the polygamma function given in \eqref{asymptotic expansion for polygamma}, we derive the following lemma, which provides the asymptotic representation for the poly-double gamma function.
\begin{lemma}
Let $n\geq 2$ be a natural number and $x>0$. Then the asymptotic expansion of $\psi_2^{(n)}(x+1)$ is given by
\begin{equation}\label{asymptotic expansion of poly-double gamma function}
\begin{aligned}
\psi_2^{(n)}(x+1)&\!=\!-x\psi^{(n)}(x+1)\!-\!(n+1)\psi^{(n-1)}(x+1)\!+\!(-1)^{n}\left(\dfrac{x^2}{2}\!+\!\dfrac{x}{2}\!+\!\dfrac{1}{12}\right)\dfrac{ n!}{x^{n+1}}\\
&\!+\!(-1)^{n+1}\left( x+\dfrac{1}{2}\right) \dfrac{(n+1)!}{nx^n}\!+\!(-1)^n\dfrac{(n+1)!}{(2n-2)x^{n-1}}\!+\!\sigma_n(x)+\tau_n(x),
\end{aligned}
\end{equation}
where $\sigma_n(x)$ and $\tau_n(x)$ are defined as
\begin{align}\label{value of sigma}
\sigma_n(x) = (-1)^{n+1} \sum_{k=1}^{N-1} B_{2k+2} \dfrac{(2k+n-1)!}{(2k+2)! x^{2k+n}}
\end{align}
\begin{align}\label{value of tau}
\tau_n(x) = (-1)^{n+1} \int_0^\infty t^{n-3} e^{-xt} \left(\dfrac{t}{e^t - 1} - \sum_{k=1}^{2N} \dfrac{B_k}{k!} t^k \right) dt, \quad N=1,2,3, \ldots .
\end{align}
and $B_k$ denotes the well-known Bernoulli numbers.
\end{lemma}
\begin{proof}
Consider the expression of $\Gamma_2$ given in \eqref{barnes g function definition}. From the logarithmic derivative of the asymptotic expression of $\Gamma_2$ using [\cite{Ferreira_2001_asym of double gamma}, Theorem 1] (see also \cite{Pedresan_2005_gamma2 remainder mediterr}), we obtain
%
\begin{equation}\label{asymptotic expansion of di double gamma function}
\begin{aligned}
	\psi_2(x+1)\! =\!-\dfrac{x}{2} \!-\! x \psi(x+1)\! -\! \log \Gamma(x+1)\! +\! \log x \left( x + \dfrac{1}{2} \right)\!+\! \dfrac{1}{x} \left( \dfrac{x^2}{2}\! +\! \dfrac{x}{2} + \dfrac{1}{12} \right)\\
	- \sum_{k=1}^{N-1} \dfrac{B_{2k+2}}{(2k+1)(2k+2) x^{2k+1}} - \int_{0}^{\infty} e^{-xt} \left(\dfrac{t}{e^t-1}-\sum_{k=0}^{2N} \dfrac{B_k}{k!}t^k\right) \dfrac{dt}{t^2}.
\end{aligned}
\end{equation}
Taking the $n^{\text{th}}$ derivative of \eqref{asymptotic expansion of di double gamma function}, we arrive at \eqref{asymptotic expansion of poly-double gamma function}. This concludes the proof.
\end{proof}
In the above Lemma, it is important to note that we obtain the asymptotic expansion of $\psi_2^{(n)}(x+1)$ by differentiating the asymptotic expansion of $\log \Gamma_2(x+1)$. This relies on the fact that the derivative of an asymptotic expansion is again an asymptotic, which is not true in general. However, it is valid here because the function $\log \Gamma_2(z+1)$ is analytic in the half-plane $\Re(z) > -1$. In \cite{Olver_asymp and special fun_nist}, it is stated that the differentiation of an asymptotic expansion is legitimate when the underlying function $f(z)$ is analytic in the complex variable $z$. In such cases, the asymptotic expansion of $f(z)$ may be differentiated any number of times in any sector that is properly interior to the original sector of validity and has the same vertex.

Now, using the recurrence relation for $\psi_2^{(n)}(x+1)$ given in \eqref{recurrence relation of poly-double gamma function}, together with its asymptotic expansion in \eqref{asymptotic expansion of poly-double gamma function}, we derive the following result, which will be used in the proof of Theorem \ref{complete monotonicity of F(x;omega) theorem}.


\begin{theorem}\label{limit formula thm}
For $n \geq 2$, the poly-double gamma function $\psi_2^{(n)}(x)$ satisfies
\begin{align}\label{asymptotic expansion for poly-double gamma function limit value}
\lim_{x \to \infty} x^{n-1} \psi_2^{(n)}(x) = (-1)^{n-1} (n-2)!, \quad x \in (0, \infty).
\end{align}
\end{theorem}
\begin{proof}
Substituting the values of $\psi^{(n)}(x+1)$ and $\psi_2^{(n)}(x+1)$  from \eqref{recurrence relation for polygamma} and \eqref{recurrence relation of poly-double gamma function}, respectively, into \eqref{asymptotic expansion of poly-double gamma function}, we obtain
\newline
$
\displaystyle
 \psi_2^{(n)}(x)-\psi^{(n)}(x)
 $
\begin{align}\label{difference of poly-double gamma and poly gamma}
=\!\! -x &\! \left( \!\! \psi^{(n)}(x) \! + \! \dfrac{(-1)^n n!}{x^{n+1}} \! \right) \!\! -\!\!(n+1)\! \! \left( \!\! \psi^{(n-1)}(x) \! + \! \dfrac{(-1)^{n-1} (n-1)!}{x^n} \! \right) \!\!
+\!(-1)^{n} \!\! \left(\!\!\dfrac{x^2}{2}-\!\dfrac{x}{2}-\! \dfrac{1}{12}\! \right)\!\dfrac{ n!}{x^{n+1}} \nonumber\\
&+(-1)^{n+1}\!\!\left(\!\! x\!+\! \dfrac{1}{2}\!\right)\!\! \dfrac{(n+1)!}{nx^n}
\!+\!(-1)^n\dfrac{(n+1)!}{(2n-2)x^{n-1}}\!+\!\sigma_n(x)\!+\!\tau_n(x),\quad n\geq 2,
\end{align}
where $\sigma_n(x)$ and $\tau_n(x)$ are defined in \eqref{value of sigma} and \eqref{value of tau} respectively. It is clear from \eqref{value of sigma} and \eqref{value of tau} that
\begin{align*}
\lim_{x \to \infty} x^{n-1} \sigma_n(x)\to 0 \quad \text{and} \quad \lim_{x \to \infty} x^{n-1} \tau_n(x) \to 0.
\end{align*}
Multiplying both sides of \eqref{difference of poly-double gamma and poly gamma} by $x^{n-1}$ and taking the limit as $x \to \infty$, we obtain
\newline
$
\displaystyle
\lim_{x \to \infty}x^{n-1} \psi_2^{(n)}(x)-  \lim_{x \to \infty} x^{n-1}  \psi^{(n)}(x)
$
\begin{align*}
=\!\!-\!\lim_{x \to \infty} x^n \psi^{(n)}(x)\!-\!(n+1)\!\!\lim_{x \to \infty} x ^{n-1}& \psi^{(n-1)}(x)+\!\!  \dfrac{ (-1)^{n}\! n!}{2}
\!+\!\dfrac{(-1)^{n+1}\!(n+1)!}{n}\!+ \!\!(-1)^n\!\dfrac{(n+1)!}{(2n-2)}\\
&+ \!\!\lim_{x \to \infty} x^{n-1} \sigma_n(x)\!+\!\lim_{x \to \infty} x^{n-1} \tau_n(x),\quad n\geq 2.
\end{align*}
Using \eqref{asymptotic expansion for polygamma limit form}, we get
\newline
$
\displaystyle
\lim_{x \to \infty}x^{n-1} \psi_2^{(n)}(x)
$
\begin{align*}
=\!\!(\!-1\!)^{n}\! (n-1)!\!+\!(-1)^{n-1}\!(n\!+\!1)(n\!-\!2)!\!+\! \dfrac{(-1)^{n} n!}{2}\!+\! \dfrac{(-1)^{n+1}(n\!+\!1)!}{n}\!+\! (-1)^n\!\dfrac{(n\!+\!1)!}{(2n\!-\!2)}, \ n\geq 2.
\end{align*}
By doing simple calculaitons, we obtain \eqref{asymptotic expansion for poly-double gamma function limit value}.
\end{proof}
The graphical representation of \eqref{asymptotic expansion for poly-double gamma function limit value} is provided below, where we have considered the case $n=2$.
\begin{figure}[H]
\footnotesize
\stackunder[5pt]{\includegraphics[scale=0.8]{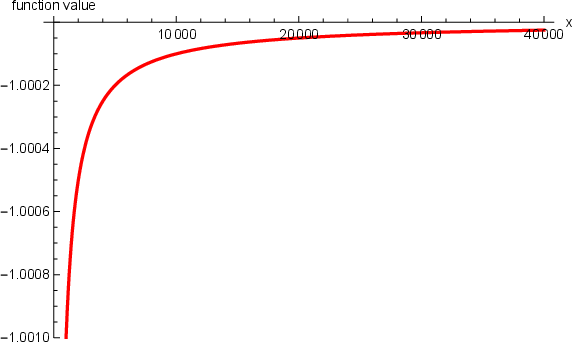}}{}
\caption{Graph of $x^{n-1}\psi_2^{(n)}(x)$ for $n=2$ and $x\in(0,40000]$}
\label{fig:limit value of poly-double gamma function}
\end{figure}
In Figure \ref{fig:limit value of poly-double gamma function}, we observe that for a large value of $x$, (let $x=40000$), as $x$ increases further, the function $x^{n-1}\psi_2^{(n)}(x)$ converges to $(-1)^{n-1} (n-2)$. For $n=2$, this evaluates to $-1$. Hence, as we can see in Figure \ref{fig:limit value of poly-double gamma function}, as $x$ approaches infinity, $x \psi_2^{(2)}(x)$ settles at $-1$.

\section{Complete monotonicity}\label{complete monotonicity}
In this section, we determine the range of parameters for which the function involving the poly-double gamma function is completely monotonic. As a result, we derive bounds for its ratio and, using these bounds, obtain various related inequalities. We begin with the following lemma, which plays a key role in proving Theorem \ref{complete monotonicity of F(x;omega) theorem}.
\begin{lemma}\label{lemma}
Let the function $f_n(t)$ be defined as
\begin{align}\label{value of f_n(t)}
f_n(t) = \dfrac{t^{n-1}}{(1-e^{-t})^2}, \quad n \geq 2.
\end{align}
Then, for all $a>0$, the integral
\begin{align}\label{value of integral I}
I_1(a;n) = \int_0^1 \left[(2n-3)x^2-1\right] f_n(a(1+x)) f_n(a(1-x)) dx,
\end{align}
is negative.
\end{lemma}
\begin{proof}
Let us define the functions
\begin{align}\label{value of g(x;n)}
g(x;n) &= \left[(2n-3)x^2-1\right](1-x^2)^{n-3},
\end{align}
and
\begin{align}\label{value of h}
h(x;a) &= \dfrac{a(1-x)}{1-e^{-a(1-x)}} \dfrac{a(1+x)}{1-e^{-a(1+x)}}.
\end{align}
Then, using \eqref{value of g(x;n)} and \eqref{value of h}, the integral $I_1(a;n)$ given in \eqref{value of integral I} can be rewritten as
\begin{align*}
I_1(a;n) = a^{2(n-3)} \int_0^1 g(x;n) \left(h(x;a)\right)^2 dx.
\end{align*}
Consider the function $v$ defined as
\begin{align}\label{value of log h}
v = \log h(x;a).
\end{align}
Substituting $h(x;a)$ from \eqref{value of h} into \eqref{value of log h}, we obtain
\begin{align*}
v = 2\log a + \log(1-x) + \log(1+x)- \log(1-e^{-a(1-x)}) - \log(1-e^{-a(1+x)}).
\end{align*}
Differentiating $v$ with respect to $x$, we get
\begin{align*}
\dfrac{\partial v}{\partial x} = -\underbrace{\left(\dfrac{2x}{1-x^2} + a \dfrac{e^{-a(1-x)} - e^{-a(1+x)}}{(1-e^{-a(1-x)})(1-e^{-a(1+x)})}\right)}_{v_1(x;a)}.
\end{align*}
	Clearly, the function $v_1(x;a)$ is positive for all $a > 0$ and $x \in (0,1)$. Thus, we conclude that $h(x;a)$ is decreasing function. Now, we have
	\begin{align*}
		\dfrac{d}{dx} \left(h(x;a)\right)^2 &= 2 h(x;a) h^{'}(x;a).
	\end{align*}
	Since $h(x;a)$ is positive and $h^{'}(x;a)$ is negative, it follows that $\left(h(x;a)\right)^2$ is also decreasing on $(0,1)$. That is
	\begin{equation}
		\begin{aligned}\label{inequalities involving h}
			\left(h(x;a)\right)^2  \geq \left(h((2n-3)^{-1/2};a)\right)^2, \quad \text{if} \quad  0 < x \leq  (2n-3)^{-1/2},\\
			\text{and} \quad  \left(h(x;a)\right)^2 \leq  \left(h((2n-3)^{-1/2};a)\right)^2, \quad \text{if} \quad (2n-3)^{-1/2} \leq  x < 1.
		\end{aligned}
	\end{equation}
	For the function $g(x;n)$ defined in \eqref{value of g(x;n)}, we have
	\begin{equation}\label{sign of g in (0,1)}
		\begin{aligned}
			g(x;n)  < 0, \quad \text{if} \quad  0 < x < (2n-3)^{-1/2},\\
			\text{and} \quad g(x;n) > 0, \quad \text{if} \quad (2n-3)^{-1/2} < x < 1.
		\end{aligned}
	\end{equation}
	Thus, using \eqref{inequalities involving h} and \eqref{sign of g in (0,1)} we obtain
	\begin{align*}
		g(x;n) (h(x;a))^2 \leq g(x;n) \left(h\left((2n-3)^{-1/2};a\right)\right)^2, \quad x\neq (2n-3)^{-1/2}, x\in (0,1).
	\end{align*}
	Integrating this inequality over $(0,1)$ yields
	\begin{align}\label{inequality involving integral of g}
		\int_0^1 g(x;n) (h(x;a))^2 dx \leq \left(h\left((2n-3)^{-1/2};a\right)\right)^2 \int_0^1 g(x;n) dx.
	\end{align}
	Now using \eqref{value of g(x;n)}, we can write
	\begin{align*}
		\int_0^1 g(x;n) dx &= \int_0^1 \dfrac{d}{dx} \left(-x(1-x^2)^{n-2}\right) dx=0.
	\end{align*}
	Therefore, \eqref{inequality involving integral of g} simplifies to
	\begin{align*}
		\int_0^1 g(x;n) (h(x;a))^2 dx \leq 0.
	\end{align*}
	Multiplying both sides by $a^{2(n-3)}$, we obtain the required result.
\end{proof}
We have provided a graphical illustration to support Lemma \ref{lemma}, as shown below.
\begin{figure}[H]
	\footnotesize
	\stackunder[5pt]{\includegraphics[scale=0.8]{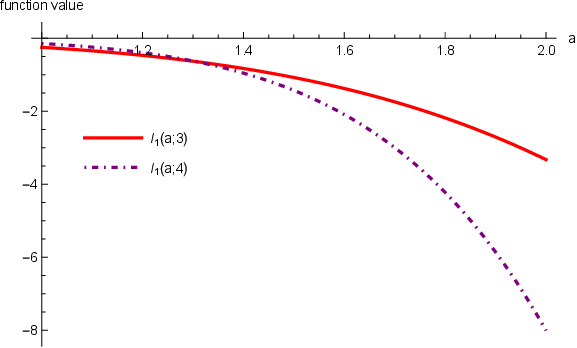}}{}
	\caption{Graph of $I_1(a;n)$ for $n=3, 4$ and $a\in(1,2)$}
	\label{fig:negativity of the integral}
\end{figure}
In Figure \ref{fig:negativity of the integral}, we plot the integral $I_1(a;n)$ for $n=3,4$ and $a\in (1,2)$. It is clear that $I_1(a;n)$ remains negative for both values of $n$. Although we demonstrated for $n=3$ and $n=4$ with $a\in(1,2)$, the result holds for all $n \geq 3$ and for any $a\in(0,\infty)$.

We are now in a position to establish the main theorem of this section, which is presented below.

\begin{theorem}\label{complete monotonicity of F(x;omega) theorem}
Let $n\in \mathbb{N}$ such that $n \geq 3$, and let $\omega \in \mathbb{R}$. Define the function
\begin{align*}
F_n(x; \omega) = (\psi_2^{(n)}(x))^2 - \omega \psi_2^{(n-1)}(x) \psi_2^{(n+1)}(x), \quad x \in (0,\infty),
\end{align*}
where $\psi_2^{(n)}(x)$ denotes the poly-double gamma function defined in \eqref{poly-double gamma function series form}. Then, the following statements hold.
\begin{enumerate}
\rm	\item \label{F_n(x;omega) cm} The function $F_n(x; \omega)$ exhibits complete monotonicity with respect to $x$ on $(0, \infty)$ if, and only if, $\omega \leq \dfrac{n - 2}{n - 1}$.
\rm	\item \label{F_n(x;omega_m) cm} The function $-F_n(x; \omega)$ exhibits complete monotonicity with respect to $x$ on $(0,\infty)$ if, and only if, $\omega \geq \dfrac{n}{n+1}$.
\end{enumerate}
\end{theorem}


\begin{proof}
Let us suppose that $F_n(x; \omega)$ is completely monotonic for all $n \geq 3$. Then, by Definition \ref{cmf definition}, we have
\begin{align*}
F_n(x; \omega) > 0 \Longleftrightarrow   \omega < \dfrac{(\psi_2^{(n)}(x))^2}{\psi_2^{(n-1)}(x) \psi_2^{(n+1)}(x)},
\end{align*}
which can be rewritten as
\begin{align*}
 \omega <  \dfrac{(x^{n-1} \psi_2^{(n)}(x))^2}{x^{n-2} \psi_2^{(n-1)}(x)  x^n \psi_2^{(n+1)}(x)}.
\end{align*}
Taking the limit as $x \to \infty$ on both sides and utilizing the asymptotic limits given in \eqref{asymptotic expansion for poly-double gamma function limit value} on the right-hand side, we obtain
\begin{align*}
\omega\leq  \lim_{x\to\infty} \dfrac{(x^{n-1} \psi_2^{(n)}(x))^2}{x^{n-2} \psi_2^{(n-1)}(x)  x^n \psi_2^{(n+1)}(x)} 	&= \dfrac{n-2}{n-1}.
\end{align*}
Now, we assume $\omega \leq \dfrac{n - 2}{n - 1}$. To prove $F_n(x; \omega)$ is completely monotonic, we rewrite the function as follows
\begin{align}\label{F_n(x;omega) in terms of n}
F_n(x; \omega ) = F_n\left(x; \dfrac{n-2}{n-1}\right) + \left(\dfrac{n-2}{n-1} - \omega \right) (-1)^n  \psi_2^{(n-1)}(x) (-1)^{n+2} \psi_2^{(n+1)}(x).
\end{align}
From Theorem \ref{poly double cm thm}, $(-1)^{n+1} \psi_2^{(n)}(x)$ is completely monotone. Consequently, the function $(-1)^n \psi_2^{(n-1)}(x) (-1)^{n+2} \psi_2^{(n+1)}(x)$ is also completely monotonic, as complete monotonicity is preserved under multiplication. Moreover, the class of completely monotonic functions forms a convex cone. Therefore, from \eqref{F_n(x;omega) in terms of n}, it is sufficient to show that $F_n\left(x; \dfrac{n-2}{n-1}\right)$ is  completely monotonic on $(0, \infty)$ in order to get the complete monotonicity of $F_n(x; \omega)$.
\begin{align}\label{F(x;omega) in terms of n-2/n-1}
F_n\left(x; \dfrac{n-2}{n-1}\right) = (\psi_2^{(n)}(x))^2 - \left(\dfrac{n-2}{n-1} \right)\psi_2^{(n-1)}(x) \psi_2^{(n+1)}(x).
\end{align}
Substituting the integral representation of $\psi_2^{(n)}(x)$ given in \eqref{integral form of poly-double gamma function} into \eqref{F(x;omega) in terms of n-2/n-1}, we obtain
\newline
$
\displaystyle
F_n\left(x; \dfrac{n-2}{n-1}\right)
$
\begin{align*}
= \underbrace{\left( \int_0^\infty e^{-xt} \dfrac{t^n}{(1 - e^{-t})^2} dt \right)^2}_{I_2(x;n)}- \left(\dfrac{n-2}{n-1}\right) \underbrace{\left(\int_0^\infty e^{-xt} \dfrac{t^{n-1}}{(1 - e^{-t})^2} dt  \int_0^\infty e^{-xt} \dfrac{t^{n+1}}{(1 - e^{-t})^2} dt \right)}_{I_3(x;n)}.
\end{align*}
Using the convolution of the Laplace transform in $I_2(x; n)$ and $I_3(x; n)$, we get

\begin{align*}
F_n\left(x; \dfrac{n-2}{n-1}\right) &= \int_0^\infty e^{-xt} \left(\int_0^t \dfrac{u^n}{(1 - e^{-u})^2} \dfrac{(t-u)^n}{(1 - e^{-(t-u)})^2} du \right) dt \\
& - \left(\dfrac{n-2}{n-1}\right) \int_0^\infty e^{-xt} \left(\int_0^t \dfrac{u^{n+1}}{(1 - e^{-u})^2} \dfrac{(t-u)^{n-1}}{(1 - e^{-(t-u)})^2} du \right) dt.
\end{align*}
Rewriting it with the help of \eqref{value of f_n(t)} gives
\begin{align}\label{value of F_n(x; n-2/n-1)}
F_n\left(x; \dfrac{n-2}{n-1}\right) &= \int_0^\infty e^{-xt} I_4(t;n) dt,
\end{align}
where
\begin{equation}\label{value of I_4}
I_4(t;n) = \int_0^t \left(t - u \dfrac{2n-3}{n-1}\right) u f_n(u) f_n(t-u) du.
\end{equation}
Substituting $u= \dfrac{t}{2}(1+x)$ into \eqref{value of I_4} and using \eqref{value of f_n(t)}, we obtain
\begin{equation}\label{value of I_2 (1)}
I_4(t;n) = \dfrac{t^3}{8(n-1)} \int_{-1}^1 \left[1 - (2n-4)x - (2n-3)x^2\right] f_n\left(\dfrac{t}{2}(1+x)\right) f_n\left(\dfrac{t}{2}(1-x)\right) dx,
\end{equation}
where
\begin{align*}
f_n\left(\dfrac{t}{2}(1+x)\right) f_n\left(\dfrac{t}{2}(1-x)\right)= \dfrac{t^{2n-2}(1-x^2)^{n-1}}{4^{n-1}\left(1-e^{-\frac{t}{2}(1+x)}-e^{-\frac{t}{2}(1-x)}+e^{-t}\right)}.
\end{align*}
Clearly, $f_n\left(\dfrac{t}{2}(1+x)\right) f_n\left(\dfrac{t}{2}(1-x)\right)$ is an even function in $x$. It follows that
\begin{align*}
\int_{-1}^1 x f_n\left(\dfrac{t}{2}(1+x)\right) f_n\left(\dfrac{t}{2}(1-x)\right) dx = 0.
\end{align*}
Thus \eqref{value of I_2 (1)} simplifies to
\begin{align*}
I_4(t;n) = & - \dfrac{t^3}{4(n-1)} \int_0^1 \left[(2n-3)x^2-1\right] f_n\left(\dfrac{t}{2}(1+x)\right) f_n\left(\dfrac{t}{2}(1-x)\right) dx.
\end{align*}
Using \eqref{value of integral I}, the given integral can be written in terms of $I_1(t/2; n)$ as
\begin{align*}
I_4(t;n) = & - \dfrac{t^3}{4(n-1)} I_1(t/2;n).
\end{align*}
By Lemma \ref{lemma}, we deduce that $I_4(t; n)>0$ for all $t > 0$ and $n \geq 3$. Hence, by applying Theorem \ref{Bernstein theorem} to \eqref{value of F_n(x; n-2/n-1)}, it follows that the function $F_n\left(x; \dfrac{n-2}{n-1}\right)$ is completely monotonic on $(0, \infty)$ for all $n \geq 3$. Consequently, $F_n(x; \omega)$ is completely monotonic for all $\omega \leq \dfrac{n-2}{n-1}$,  which completes the proof of first part.

To prove the second part of Theorem \ref{complete monotonicity of F(x;omega) theorem}, suppose that $-F_n(x; \omega)$ is completely monotonic. Then, by Definition \ref{cmf definition}, it must be positive, which implies that
\begin{align}\label{poly double gamma less than omega_m}
\omega>	\dfrac{(\psi_2^{(n)}(x))^2}{\psi_2^{(n-1)}(x) \psi_2^{(n+1)}(x)}.
\end{align}
Applying the recurrence relation for $\psi_2^{(n)}(x)$ from \eqref{recurrence relation of poly-double gamma function} to the left-hand side of \eqref{poly double gamma less than omega_m}, we obtain
\begin{align*}
\omega > \dfrac{( \psi_2^{(n)}(x+1)+\psi^{(n)}(x))^2}{(\psi_2^{(n-1)}(x+1)+\psi^{(n-1)}(x))(\psi_2^{(n+1)}(x+1)+\psi^{(n+1)}(x))}.
\end{align*}
Next, using the recurrence relation for the polygamma function from \eqref{recurrence relation for polygamma} on the left-hand side and taking the limit as $x \to 0$, we obtain
\begin{align*}
\omega\geq \lim_{x \to 0}	\dfrac{(\psi_2^{(n)}(x))^2}{\psi_2^{(n-1)}(x) \psi_2^{(n+1)}(x)} =  \dfrac{n}{n+1}.
\end{align*}

Conversely, let us take $\omega \geq \dfrac{n}{n+1}$, and aim is to show that $-F_n(x; \omega)$ is completely monotonic. We begin by expressing $-F_n(x; \omega)$ as
\begin{align}\label{-F_n(x;omega) in terms of n}
-F_n(x; \omega) = -F_n\left(x; \dfrac{n}{n+1}\right) + \left(\omega - \dfrac{n}{n+1}\right) (-1)^n \psi_2^{(n-1)}(x)  (-1)^n \psi_2^{(n+1)}(x).
\end{align}
As complete monotonicity remains preserved under multiplication, Theorem \ref{poly double cm thm} confirms the complete monotonicity of $(-1)^n \psi_2^{(n-1)}(x)  (-1)^n \psi_2^{(n+1)}(x)$. Further the class of completely monotone functions constitutes a convex cone, it follows from \eqref{-F_n(x;omega) in terms of n} that to prove the complete monotonicity of $-F_n(x; \omega)$, it is sufficient to prove that $-F_n\left(x; \dfrac{n}{n+1}\right)$ is completely monotone. Thus, we have
\begin{align}\label{value of -F(x,n/n+1)}
-F_n\left(x; \dfrac{n}{n+1}\right) = \left(\dfrac{n}{n+1}\right) \psi_2^{(n-1)}(x) \psi_2^{(n+1)}(x) - (\psi_2^{(n)}(x))^2.
\end{align}
Using \eqref{poly-double gamma function series form} in \eqref{value of -F(x,n/n+1)}, we get
\newline
$
\displaystyle
-F_n\left(x; \dfrac{n}{n+1}\right)
$
\begin{align*}
=\!\! \left(\dfrac{n}{n+1}\right) \!\left(\!(n-1)! \sum_{k=0}^{\infty}\dfrac{(1+k)}{(x+k)^{n}} \! \right) \! \left( \! (n+1)! \sum_{k=0}^{\infty}\dfrac{(1+k)}{(x+k)^{n+2}} \! \right) \! - \left(\! n! \sum_{k=0}^{\infty}\dfrac{(1+k)}{(x+k)^{n+1}} \! \right)^2.
\end{align*}
By a simple computation, we obtain
\begin{align*}
-F_n\left(x; \dfrac{n}{n+1}\right) = (n!)^2\left(\sum_{k=0}^{\infty}\dfrac{(1+k)}{(x+k)^{n}} \sum_{k=0}^{\infty}\dfrac{(1+k)}{(x+k)^{n+2}} - \left(\sum_{k=0}^{\infty}\dfrac{(1+k)}{(x+k)^{n+1}}\right)^2\right).
\end{align*}
Using Lagrange's identity given in \eqref{lagrange identity}, the expression simplifies to
\begin{align*}
-F_n\left(x; \dfrac{n}{n+1}\right) = (n!)^2 \sum_{k=0}^{\infty} \sum_{j=k+1}^{\infty} (k-j)^2 (1+k)(1+j)[(x+k)(x+j)]^{-n-2}.
\end{align*}
Differentiating up to the $p$th order yields
\newline
$
\displaystyle
(-1)^p \dfrac{d^p}{dx^p} \left(-F_n\left(x; \dfrac{n}{n+1}\right)\right)
$
\begin{align*}
=\!(n!)^2 \!\sum_{k=0}^{\infty}\! \sum_{j=k+1}^{\infty} \! \! \! (k-j)^2 \! \sum_{r=0}^{p}\! \! \binom{p}{r} \! (x+k)^{-n-2-r} (x+j)^{-n-2-p+r}\!\!\left(\!\prod_{s=0}^{r-1} (n+2+s) \!\! \!\! \prod_{s=0}^{p-r-1} \!\!\! (n+2+s) \!\! \right)\!\!.
\end{align*}

The right-hand side of above equation is always positive, ensuring the complete monotonicity of the function $-F_n\left(x; \dfrac{n}{n+1}\right)$, which completes the proof of second part.
\end{proof}
The graphical interpretation of the first and second parts of Theorem \ref{complete monotonicity of F(x;omega) theorem} is presented below.
\begin{figure}[H]
\centering
\begin{minipage}{0.48\textwidth}
\centering
\includegraphics[scale=0.79]{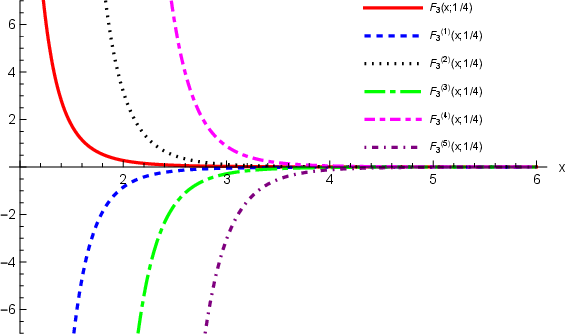}
\footnotesize
Graph of $F_3(x; 1/4)$ and its derivatives.
\caption{Complete monotonicity of $F_3(x; 1/4)$}
\label{fig:Complete monotonicity of $F_3(x; 1/4)$}
\end{minipage} \hfil
\begin{minipage}{0.48\textwidth}
\centering

\includegraphics[scale=0.79]{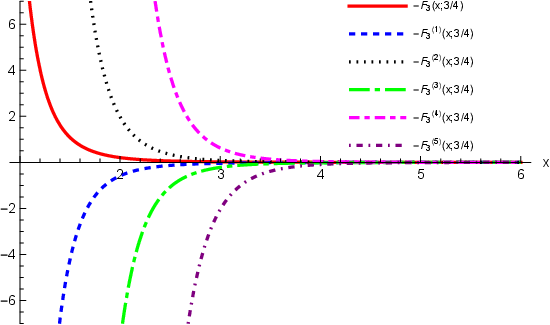}
\footnotesize

Graph of $-F_3(x; 3/4)$ and its derivatives
\caption{Complete monotonicity of $-F_3(x; 3/4)$}
\label{fig:Complete monotonicity of $-F_3(x; 3/4)$}
\end{minipage}
\end{figure}
For $n = 3$ and $\omega \leq \dfrac{1}{2}$, Figure \ref{fig:Complete monotonicity of $F_3(x; 1/4)$} demonstrates the complete monotonicity of $F_3(x; 1/4)$, as it is positive and its derivatives alternate in sign, with the first derivative being negative (as per Definition \ref{cmf definition}). Similarly, for $n = 3$ and $\omega = \dfrac{3}{4}$, Figure \ref{fig:Complete monotonicity of $-F_3(x; 3/4)$} shows the complete monotonicity of $-F_3(x; 3/4)$.
%
\begin{corollary}\label{upper-lower bound polydouble gamma corollary}
For the poly-double gamma function $\psi_2^{(n)}(x)$, defined in \eqref{poly-double gamma function series form}, the following inequality holds
\begin{align}\label{lower-upper bound poly-double gamma function}
\dfrac{n-2}{n-1}<\dfrac{(\psi_2^{(n)}(x))^2}{\psi_2^{(n-1)}(x) \psi_2^{(n+1)}(x)} < \dfrac{n}{n+1}, \quad n \geq 3, \  x > 0.
\end{align}
Moreover, both upper and lower bounds are sharp.
\end{corollary}
\begin{proof}
Since complete monotonicity of any function implies it's positivity, the complete monotonicity of $F_n\left(x; \dfrac{n-2}{n-1}\right)$ and $-F_n\left(x; \dfrac{n}{n+1}\right)$, established in Theorem \ref{complete monotonicity of F(x;omega) theorem}, leads to the desired result. The sharpness of bounds follows from the necessary and sufficient conditions established in first and second parts of the Theorem \ref{complete monotonicity of F(x;omega) theorem}.
\end{proof}

\begin{corollary}\label{cauchy schwarz cor}
For any $x > 0$, the following inequalities hold.
\begin{enumerate}
\rm \item For integers $n \geq 3$,
\begin{align}\label{cauchy schwarz}
\left(\sum_{k=0}^{\infty}\dfrac{k+1}{(x+k)^{n+1}} \right)^2 < \left(\sum_{k=0}^{\infty}\dfrac{k+1}{(x+k)^n} \right)\left(\sum_{k=0}^{\infty}\dfrac{k+1}{(x+k)^{n+2}} \right).
\end{align}
\rm \item For integers $n \geq  3$,
\begin{align}\label{reversed cauchy schwarz}
\left(\sum_{k=0}^{\infty}\dfrac{k+1}{(x+k)^{n+1}} \right)^2 > \dfrac{n^2 - n - 2}{n^2 - n}\left(\sum_{k=0}^{\infty}\dfrac{k+1}{(x+k)^n} \right)\left(\sum_{k=0}^{\infty}\dfrac{k+1}{(x+k)^{n+2}} \right).
\end{align}
\end{enumerate}
\end{corollary}
\begin{proof}
Substituting the series representation of $\psi_2^{(n)}(x)$ given in \eqref{poly-double gamma function series form} into the right-hand side of the inequality in \eqref{lower-upper bound poly-double gamma function}, we obtain \eqref{cauchy schwarz}. Similarly, substituting the same series into the left-hand side of \eqref{lower-upper bound poly-double gamma function}, we get \eqref{reversed cauchy schwarz}.
\end{proof}

In Corollary \ref{cauchy schwarz cor}, the inequality \eqref{cauchy schwarz} corresponds to the classical Cauchy–Schwarz inequality, whereas the inequality \eqref{reversed cauchy schwarz} can be viewed as a reverse-type Cauchy–Schwarz inequality.

\begin{theorem}\label{inequality polydouble gamma thm}
Let the function $G_n(x; r)$ is defined by
\begin{align} \label{G_n(x;r) equation}
G_n(x; r) = \left((-1)^{n+1} \psi_2^{(n)}(x)\right)^r, \quad x \in (0, \infty), \quad n \geq 3,
\end{align}
where $\psi_2^{(n)}(x)$ denotes the poly-double gamma function given in \eqref{poly-double gamma function series form}. Then the following statements hold.
\begin{enumerate}
\rm \item \label{convexity} The function $G_n(x; r)$ is strictly convex for $r \in \left(-\infty, -\dfrac{1}{n-1}\right) \cup (0, \infty)$ and strictly concave for $r \in \left(-\dfrac{1}{n+1}, 0\right)$.
\rm \item \label{subadditivity} For $r \in \left(-\infty, -\dfrac{1}{n-1}\right)$,
\begin{align}\label{subadditivity equation}
\left((-1)^{n+1} \psi_2^{(n)}(x)\right)^r+ \left((-1)^{n+1} \psi_2^{(n)}(x)\right)^r<  \left((-1)^{n+1} \psi_2^{(n)}(x+y)\right)^r.
\end{align}
\rm \item \label{superadditivity} For $r \in \left(-\dfrac{1}{n+1}, 0\right)$,
\begin{align}\label{superadditivity equation}
\left((-1)^{n+1} \psi_2^{(n)}(x)\right)^r+ \left((-1)^{n+1} \psi_2^{(n)}(x)\right)^r > \left((-1)^{n+1} \psi_2^{(n)}(x+y)\right)^r.
\end{align}
\end{enumerate}
\end{theorem}


\begin{proof}
Differentiating $G_n(x;r)$ twice with respect to $x$, we obtain
\begin{align*}
G_n''(x;r) = r \left((-1)^{n+1}\psi_2^{(n)}(x)\right)^{r-2} \left( (r-1)(\psi_2^{(n+1)}(x))^2 + \psi_2^{(n)}(x)\psi_2^{(n+2)}(x) \right).
\end{align*}
Furthermore, we can rearrange this expression as
\begin{align} \label{double derivative of G_n}
\dfrac{1}{r} \left(\psi_2^{(n+1)}(x)\right)^{-2} \left(G_n(x;r)\right)^{-1} \left(G_n(x;r)\right)^{2/r} G_n''(x;r) = (r-1) +  \dfrac{\psi_2^{(n)}(x)\psi_2^{(n+2)}(x)}{\left(\psi_2^{(n+1)}(x)\right)^2}.
\end{align}
To determine the sign of $G_n''(x; r)$, we need to analyze the sign of $\dfrac{1}{r}$ and sign of the right-hand side of \eqref{double derivative of G_n}, as all other terms are positive regardless of the choices of $r$, $x$, and $n$.

Using \eqref{lower-upper bound poly-double gamma function}, the right-hand side of \eqref{double derivative of G_n} satisfies
\begin{align}
(r-1) + \dfrac{\psi_2^{(n)}(x)\psi_2^{(n+2)}(x)}{\left(\psi_2^{(n+1)}(x)\right)^2} &< (r-1) + \dfrac{n}{n-1}, \label{upper bound poly double gamma modified}\\
(r-1) + \dfrac{\psi_2^{(n)}(x)\psi_2^{(n+2)}(x)}{\left(\psi_2^{(n+1)}(x)\right)^2} &> (r-1) + \dfrac{n+2}{n+1}. \label{lower bound poly double gamma modified}
\end{align}
Now, if we choose $r < -\dfrac{1}{n-1}$, then by using \eqref{upper bound poly double gamma modified}, the right-hand side of \eqref{double derivative of G_n} is negative. Since $\dfrac{1}{r}$ is also negative for $r < -\dfrac{1}{n-1}$, it follows that $G_n''(x; r) > 0$, implying that $G_n(x; r)$ is strictly convex when $r < -\dfrac{1}{n-1}$.

Similarly, if we choose $r > -\dfrac{1}{n+1}$, then by \eqref{lower bound poly double gamma modified}, the right-hand side of \eqref{double derivative of G_n} is positive. In this case, the sign of $\dfrac{1}{r}$ depends on $r$: it is negative for $r \in \left(-\dfrac{1}{n+1}, 0\right)$ and positive for $r > 0$. Therefore, $G_n(x; r)$ is strictly concave when $r \in \left(-\dfrac{1}{n+1}, 0\right)$ and convex when $r > 0$, which completes the proof of first part.

Next, note that the strict convexity of a function $f(x)$, together with the condition $\lim_{x \to 0} f(x) = 0$, implies the superadditivity property. Therefore, the functions $G_n(x; r)$ and $-G_n(x; r)$ satisfy the superadditivity property for $r < -\dfrac{1}{n-1}$ and $r \in \left(-\dfrac{1}{n+1}, 0\right)$, respectively. This means,
\begin{align}\label{supperadditivity of G}
G_n(x; r)+G_n(y; r)<G_n(x+y; r), \quad r < -\dfrac{1}{n-1},
\end{align}
and
\begin{align}\label{supperadditivity of -G}
G_n(x; r)+G_n(y; r)>G_n(x+y; r), \quad r \in \left(-\dfrac{1}{n+1}, 0\right).
\end{align}
By substituting the expression for $G_n(x; r)$ from \eqref{G_n(x;r) equation} into \eqref{supperadditivity of G} and \eqref{supperadditivity of -G}, we obtain \eqref{subadditivity equation} and \eqref{superadditivity equation}, thus establishing parts (\ref{subadditivity}) and (\ref{superadditivity}) of the result.
\end{proof}





\section{Concluding remarks}
Motivated by the study of the complete monotonicity of gamma and polygamma functions, as discussed in \cite{Alzer_1998_Inequality polygamma_siam, Berg_2001_cm related to gamma, Ismail_2007_cm of determinants, Pedresan_2009_cm of gamma2 jmaa}, this work explores the complete monotonicity of functions defined in terms of the poly-double gamma function.
By establishing complete monotonicity results for these functions, we derive sharp upper and lower bounds, which further lead to results concerning convexity and certain inequalities for the poly-double gamma function. In the process, we also obtain a few results related to sub (super) additivity.

Since the complete monotonicity of Turan-type functions plays a crucial role in the theory of special functions \cite{Alzer_1998_Inequality polygamma_siam, Liang_cm of polygamma_JIA}, the complete monotonicity of Turan determinants of special functions have also been extensively investigated in the literature; see \cite{Baricz_cm of turan det_constructive_2013, Baricz_cm of turan det_MIA_2014,Ismail_2007_cm of determinants}. In this framework, we observe that the Turan determinant of $\psi_2^{(n)}(x)$ also exhibits the complete monotonicity property, as stated below.

\begin{proposition}\label{cm of determinant of poly double gamma}
Let $D_m(y)$ denote the determinant of order $(m+1)$ whose entries are given by the $\psi_2^{(n)}(y)$, and is defined as follows
\begin{align*}
D_{m+1}(y)
&= \det\! \begin{pmatrix}
\psi_2^{(n)}(y)      & \psi_2^{(n+j)}(y)        & \cdots & \psi_2^{(n+mj)}(y) \\
\psi_2^{(n+j)}(y)    & \psi_2^{(n+2j)}(y)       & \cdots & \psi_2^{(n+(m+1)j)}(y) \\
\psi_2^{(n+2j)}(y)   & \psi_2^{(n+3j)}(y)       & \cdots & \psi_2^{(n+(m+2)j)}(y) \\
\vdots               & \vdots                   &  & \vdots \\
\psi_2^{(n+mj)}(y)   & \psi_2^{(n+(m+1)j)}(y)   & \cdots & \psi_2^{(n+2mj)}(y)
\end{pmatrix}.
\end{align*}
Then, for $n\geq 2$, $(-1)^{(n+1)(m+1)} D_n(y)$ is completely monotonic in $y$ on $(0,\infty)$ for every positive integer $m$ and $j$.
\begin{proof}
This result follows directly from [\cite{Ismail_2007_cm of determinants}, Remark 2.9].
\end{proof}
\end{proposition}
\begin{remark}
If we take $j=1$, and $m=1$, Proposition \ref{cm of determinant of poly double gamma} gives that
\begin{align*}
(-1)^{(n+1)}\left(\psi_2^{(n)}(y)\psi_2^{(n+2)}(y)-(\psi_2^{(n)}(y))^2\right), \quad n\geq 2,
\end{align*}
is completely monotonic on $(0,\infty)$.
\end{remark}

An immediate question that arises is about the spectral behaviour of the matrix corresponding to the determinant given in Proposition \ref{cm of determinant of poly double gamma}


On the other hand, if we compare Theorem 2.1 of \cite{Alzer_1998_Inequality polygamma_siam} with Theorem \ref{complete monotonicity of F(x;omega) theorem}, the differences in parameter conditions between the polygamma and poly-double gamma functions become apparent. For future research, it would be interesting to explore whether the complete monotonicity property can be extended to the logarithmic derivative of triple gamma functions and its higher derivatives. However, challenges are expected in obtaining results for triple gamma functions. Nevertheless, results for the modified expansions similar to the one given in \cite{Das_swami_2017_Pick function} are possible. If so, it would be worthwhile to examine how the corresponding parameter conditions differ from those for the polygamma and poly-double gamma cases.

%
%
%
%


\begin{thebibliography}{9}
	

\bibitem{Adamchik_1998_zeta_analysis}
V.~S. Adamchik\ and\ H.~M. Srivastava, Some series of the zeta and related functions, Analysis (Munich) {\bf 18} (1998), no.~2, 131--144.

\bibitem{Alzer_analysis mathematica_2025}
H. Alzer and H.~L. Pedersen, Inequalities for $1/(1-\cos(x) )$ and its derivatives, Anal. Math. {\bf 51} (2025), no.~1, 63--73.

\bibitem{Alzer_1998_Inequality polygamma_siam}
H. Alzer\ and\ J.~H. Wells, Inequalities for the polygamma functions, SIAM J. Math. Anal. {\bf 29} (1998), no.~6, 1459--1466.

\bibitem{Baricz_cm of turan det_constructive_2013}
A. Baricz\ and\ M.~E. H. Ismail, Tur\'{a}n type inequalities for Tricomi confluent hypergeometric functions, Constr. Approx. {\bf 37} (2013), no.~2, 195--221.

\bibitem{Baricz_cm of turan det_MIA_2014}
A. Baricz\ and\ T.~K. Pog\'{a}ny, Functional inequalities for the Bickley function, Math. Inequal. Appl. {\bf 17} (2014), no.~3, 989--1003.


	
\bibitem{Barnes_1899_genesis of double gamma}	
E.~W. Barnes, The Genesis of the Double Gamma Functions, Proc. Lond. Math. Soc. {\bf 31} (1899), 358--381.


\bibitem{Berg_2024_cbf related to gamma}
C. Berg, A complete Bernstein function related to the fractal dimension of Pascal’s pyramid modulo a prime, Expositiones Mathematicae (2025).


\bibitem{Berg_2001_cm related to gamma}
C. Berg\ and\ H.~L. Pedersen, A completely monotone function related to the gamma function, J. Comput. Appl. Math. {\bf 133} (2001), no.~1-2, 219--230.











\bibitem{Abramowitz_handbook special func}
H. Bj\"{o}rk, Table errata: {\it Handbook of mathematical functions with formulas, graphs, and mathematical tables} (Nat. Bur. Standards, Washington, D. C., 1964) edited by Milton Abramowitz and Irene A. Stegun, Math. Comp. {\bf 23} (1969), no.~107, 691.

\bibitem{Choi_multiple gamma_amc_2003}
J. Choi, H.~M. Srivastava\ and\ V.~S. Adamchik, Multiple gamma and related functions, Appl. Math. Comput. {\bf 134} (2003), no.~2-3, 515--533.


\bibitem{Das_2018_cr acad}
S. Das, Inequalities involving the multiple psi function, C. R. Math. Acad. Sci. Paris {\bf 356} (2018), no.~3, 288--292.

\bibitem{Das_swami_2017_Pick function}
S. Das, H.~L. Pedersen\ and\ A. Swaminathan, Pick functions related to the triple Gamma function, J. Math. Anal. Appl. {\bf 455} (2017), no.~2, 1124--1138.

\bibitem{Donoghue_monotone matrix func}
W.~F. Donoghue Jr., {\it Monotone matrix functions and analytic continuation}, Die Grundlehren der mathematischen Wissenschaften, Band 207, Springer, New York, 1974.




\bibitem{Ferreira_2001_asym of double gamma}	
C. Ferreira\ and\ J. L\'{o}pez~Garc\'{\i}a, An asymptotic expansion of the double gamma function, J. Approx. Theory {\bf 111} (2001), no.~2, 298--314.


\bibitem{Ismail_2007_cm of determinants}
M.~E. H. Ismail\ and\ A. Laforgia, Monotonicity properties of determinants of special functions, Constr. Approx. {\bf 26} (2007), no.~1, 1--9.






\bibitem{Pedresan_2009_cm of gamma2 jmaa}
S. Koumandos\ and\ H.~L. Pedersen, Completely monotonic functions of positive order and asymptotic expansions of the logarithm of Barnes double gamma function and Euler's gamma function, J. Math. Anal. Appl. {\bf 355} (2009), no.~1, 33--40.

\bibitem{Liang_cm of polygamma_JIA}
L. C. Liang, L.~F. Zheng\ and\ A.~Y. Wan, A class of completely monotonic functions involving the polygamma functions, J. Inequal. Appl. {\bf 2022}, Paper No. 12, 16 pp.

\bibitem{Mezo_2017_Zeros of gamma2 itsf}
I. Mezo\ and\ M.~E. Hoffman, Zeros of the digamma function and its Barnes $G$-function analogue, Integral Transforms Spec. Funct. {\bf 28} (2017), no.~11, 846--858.




\bibitem{Mitrinovic_1970_Analytic ineq}
D.~S. Mitrinovi\'{c}, {\it Analytic inequalities}, Die Grundlehren der mathematischen Wissenschaften, Band 165, Springer, New York, 1970.


\bibitem{Nikeghbali_2009_barnes g in probability}
A. Nikeghbali\ and\ M. Yor, The Barnes $G$ function and its relations with sums and products of generalized gamma convolution variables, Electron. Commun. Probab. {\bf 14} (2009), 396--411. 	

\bibitem{Olver_asymp and special fun_nist}
F.~W.~J. Olver, {\it Asymptotics and special functions}, Computer Science and Applied Mathematics, Academic Press, New York, 1974.


\bibitem{Pedersen_2003_Pickfunc gamma2}
H.~L. Pedersen, The double gamma function and related Pick functions, J. Comput. Appl. Math. {\bf 153} (2003), no.~1-2, 361--369.

\bibitem{Pedresan_2005_gamma2 remainder mediterr}
H.~L. Pedersen, On the remainder in an asymptotic expansion of the double gamma function, Mediterr. J. Math. {\bf 2} (2005), no.~2, 171--178.



\bibitem{Quine_1996_gamma2 on mathematical phy}
J. Quine\ and\ J. Choi, Zeta regularized products and functional determinants on spheres, Rocky Mountain J. Math. {\bf 26} (1996), no.~2, 719--729.

\bibitem{Schilling_2012_bernstein func}
R. Schilling, R. Song\ and\ Z. Vondra\v{c}ek, {\it Bernstein functions}, second edition, De Gruyter Studies in Mathematics, 37, de Gruyter, Berlin, 2012.



\bibitem{trimble_1989_scm application_siam}
S.~Y. Trimble, J.~H. Wells\ and\ F.~T. Wright, Superadditive functions and a statistical application, SIAM J. Math. Anal. {\bf 20} (1989), no.~5, 1255--1259.










\bibitem{Vigneras_1979_gamma2 integral}
M. F. Vign\'{e}ras, L'\'{e}quation fonctionnelle de la fonction z\^{e}ta de Selberg du groupe modulaire ${\rm PSL}(2,\,{\bf Z})$, Ast\'{e}risque No. 61 (1979), 235--249.



\bibitem{Widder_1941_The Laplace Transform}
D.~V. Widder, {\it The Laplace Transform}, Princeton Mathematical Series, vol. 6, Princeton Univ. Press, Princeton, NJ, 1941.

\bibitem{Wimp_1981_cm in numerical}
J. Wimp, {\it Sequence transformations and their applications}, Mathematics in Science and Engineering, 154, Academic Press, Inc., New York, 1981.

\bibitem{Zhang_2020_cm of k-polygamma_JIA}
J. Zhang, L. Yin\ and\ H.~L. You, Complete monotonicity and inequalities related to generalized $k$-gamma and $k$-polygamma functions, J. Inequal. Appl. {\bf 2020}, Paper No. 21, 12 pp.

















\end{thebibliography}
\end{document}